%% file: Tanglegrams.tex
\documentclass{amsart}

\usepackage[utf8]{inputenc}
\usepackage[T1]{fontenc}
\usepackage{amsfonts,amsmath,amssymb,amsthm}
\usepackage[shortlabels]{enumitem}
\usepackage{graphicx}
\usepackage[svgnames]{xcolor}
\usepackage{tikz}
\usepackage{url}
\usepackage{hyperref}
\usepackage[pagewise]{lineno}
\usepackage{todonotes}
\usepackage{setspace}
\usepackage{algorithm}
\usepackage{algorithmicx}
\usepackage{algpseudocode}
\algnewcommand\algorithmicinput{\textbf{Input:}}
\algnewcommand\Input{\item[\algorithmicinput]}
\algnewcommand\algorithmicoutput{\textbf{Output:}}
\algnewcommand\Output{\item[\algorithmicoutput]}

\newtheorem{theorem}{Theorem}

\theoremstyle{definition}
\newtheorem{definition}{Definition}

\theoremstyle{remark}

\newcommand{\crt}{\operatorname{crt}}

\newcommand{\cT}{\mathcal{T}}

\newcommand{\cD}{\mathcal{D}}

\title[Crossing number and the tangle crossing number analogies]{Analogies between the crossing number and the tangle crossing number}

\author[Anderson et al.]{Robin Anderson}
\address[Anderson]{Department of Mathematics, Saint Louis University, St. Louis, MO 63103}
\email{rander43@slu.edu}

\author[]{Shuliang Bai}
\address[Bai, Czabarka, Sz\'ekely, Whitlatch]{Department of Mathematics, University of South Carolina, Columbia, SC 29208}
\email{\{sbai, czabarka,  szekely, hww\}@math.sc.edu}

\author[]{Fidel Barrera-Cruz}
\address[Barrera-Cruz, Smith]{School of Mathematics, Georgia Institute of Technology, Atlanta, GA 30332}
\email{\{fidelbc, heather.smith\}@math.gatech.edu}

\author[]{\'Eva Czabarka}

\author[]{Giordano Da Lozzo}
\address[Da Lozzo]{Department of Computer Science, University of California, Irvine, CA 92697}
\email{gdalozzo@uci.edu}

\author[]{Natalie L. F. Hobson}
\address[Hobson]{Department of Mathematics and Statistics, Sonoma State University, Rohnert Park, CA 94928}
\email{hobsonn@sonoma.edu}

\author[]{Jephian C.-H. Lin}
\address[Lin]{Department of Mathematics, Iowa State University, Ames, IA
50011}
\email{chlin@iastate.edu}

\author[]{Austin Mohr}
\address[Mohr]{Department of Mathematics, Nebraska Wesleyan University, Lincoln, NE 68504}
\email{amohr@nebrwesleyan.edu}

\author[]{Heather C. Smith$^*$}

\author[]{L\'aszl\'o A. Sz\'ekely}

\author[]{Hays Whitlatch}

\date{\today}

\begin{document}
\let\thefootnote\relax\footnotetext{$^*$Corresponding author}

%\linenumbers

\begin{abstract}
Tanglegrams are special graphs that consist of a pair of rooted binary trees with the same number of leaves, and a perfect matching between the
two leaf-sets. These objects are of use in phylogenetics and are represented with straightline drawings where the leaves of the two plane binary trees are on two parallel lines and only the matching edges can cross. The tangle crossing number of a tanglegram is the minimum crossing number over all such drawings and is related to biologically relevant quantities, such as the number of times a parasite switched hosts.

Our main results for tanglegrams which parallel known theorems for crossing numbers are as follows.
The removal of a single matching edge in a tanglegram with $n$ leaves decreases the tangle crossing number by at most $n-3$, and this is sharp. 
Additionally, if $\gamma(n)$ is the maximum
tangle crossing number of a tanglegram with $n$ leaves, we prove
$\frac{1}{2}\binom{n}{2}(1-o(1))\le\gamma(n)<\frac{1}{2}\binom{n}{2}$. 
Further, we provide an algorithm for computing non-trivial lower bounds on the tangle crossing number in $O(n^4)$ time. This lower bound may be tight, even for tanglegrams with tangle crossing number $\Theta(n^2)$.
\end{abstract}


\maketitle

\section{Introduction}
\label{sec:intro}
A drawing $\cD(G)$ of a graph $G$ in the plane is a set of distinct points in the plane,
one for each vertex of $G$, and a collection of simple open arcs, one for each
edge of the graph, such that if $e$ is an edge of $G$ with endpoints $v$ and $w$, then the
closure (in the plane) of the arc $\alpha$ representing $e$ consists precisely of $\alpha$ and
the two points representing $v$ and $w$. We further require that no edge--arc
intersects any vertex point. The (standard) crossing number $\operatorname{cr}(\cD(G))$ of $\cD(G)$ is the number of
pairs $(x, \{\alpha,\beta\})$, where $x$ is a point of the plane, $\alpha,\beta$ are arcs of $\cD$
representing distinct edges of $G$ such that $x\in\alpha\cap\beta$.
The crossing number $\operatorname{cr}(G)$ of a graph $G$ is defined
to be the minimum crossing number over all of its drawings.

Tanglegrams are
well studied in the phylogenetics and computer science literature. A
tanglegram of size $n$ is a triplet containing two rooted binary trees ($L$ and $R$), each
with $n$ leaves, and a fixed perfect matching $M$ between the two set of leaves.
Two tanglegrams $T_1=(L_1,R_1,M_1)$ and $T_2=(L_2,R_2,M_2)$ are the same
if there is a pair of tree-isomorphisms $(\phi,\psi)$ from $L_1$ to $L_2$ and from $R_1$ to $R_2$ that
map each pair of matched leaves to a pair of matched leaves.
A layout of a
tanglegram is a straight-line plane drawing of the trees, the first
drawn in the half plane $x\leq 0$ with its leaves on the line $x=0$
and the second in the half plane $x\geq 1$ with its leaves on the line
$x=1$, with a straight-line drawing of the matching edges between the
leaves.  The tangle crossing number $\crt(T)$ of a tanglegram $T$ is the the minimum crossing number over all
of its layouts, i.e. the minimum number of unordered pairs of
crossing edges over all layouts. The tangle crossing number
is related to the number of
times parasites switch hosts~\cite{HN} as well as the number of
horizontal gene transfers~\cite{BT}.

Though tangle crossing numbers are crossing numbers of a very specific kind of drawing of a very specific class of graphs,
a number of analogies are known between tangle crossing numbers and crossing numbers.
As with the crossing numbers of general graphs~\cite{GJ}, computing the tangle
crossing number is NP-hard~\cite{FKP}, even when both trees are
complete binary trees~\cite{BBBN}.
Testing whether a graph is planar can be done in polynomial (in fact linear) time~\cite{HT74}. Analogously, testing for tangle crossing number 0 can also be done in linear time~\cite{FKP}. Recently, Czabarka, Sz\'ekely, and
Wagner~\cite{CSW2} gave an analogue of Kuratowski's Theorem~\cite{Kura} for tanglegrams, characterizing tangle crossing number 0.
Clearly, for a graph $G$ with $e$ edges we have $\operatorname{cr}(G)=O(e^2)$, while for a tanglegram $T$ of size $n$,
$\crt(T)=O(n^2)$.
The expected crossing number of an Erd\H{o}s-R\'enyi random graph $G\in G(n,p)$ for $p=\frac{c}{n}$ for any $c>1$
is $\Theta(e^2)$ where $e=p\binom{n}{2}$ is the expected number of edges~\cite{SpT},
and the expected tangle crossing number of a random and uniformly
selected tanglegram with $n$ leaves is $\Theta(n^2)$~\cite{CSW}, i.e. both of these quantities are as large as possible in order of
magnitude.

We continue the study of the tangle crossing number with results which
parallel results for graph crossing numbers. 
Hlin\v{e}n\'{y} and Salazar \cite{HlSa} studied the crossing number of
1-edge planar graphs (i.e. graphs in which there exists an edge whose removal results in
a planar graph). For each $k\geq 1$, they define a 1-edge planar graph
$G_k$ with $2k+4$ vertices, $6k+7$ edges, and crossing number $k$. We
find that the behavior is quite similar for the tangle crossing number.  First we establish an upper bound for $\crt(T)-\crt(T-e)$ given any
tanglegram and any matching edge $e$. Then for each $n\geq 4$, we
define a tanglegram of size $n$ with tangle crossing number $n-3$ for
which there is a single matching edge whose removal yields a planar
subtanglegram. In summary, we prove the following theorem in
Section~\ref{sec:nminusthree}:

\begin{theorem}\label{thm:edgerem}
  For any tanglegram $T$ of size $n\geq 3$ and any matching edge $e$
  in $T$, let $T-e$ be the tanglegram induced by deleting the
  endpoints of $e$ and suppressing their (now degree two) neighbors. Then $\crt(T) - \crt(T-e) \leq n-3$. This
  inequality is best possible, even when $T-e$ is planar.
\end{theorem}

We then examine the largest tangle crossing number
of a tanglegram of size $n$ (an analogue of the crossing number of the complete graph on $n$ vertices).
It is well known (e.g. by the crossing
lemma or by the counting method) that the crossing number of the complete graph $K_n$ is
$\Theta(n^4)=\Theta\left(\binom{n}{2}^2\right)$.  We prove the following result in Section~\ref{sec:onehalf}:

\begin{theorem}\label{thm:limit}
  For any tanglegram $T$ of size $n$,
  $\crt(T)<\frac{1}{2}\binom{n}{2}$. If $\gamma(n)$ is the maximum
  tangle crossing number among all tanglegrams of size $n$, then
\[\lim_{n\rightarrow\infty} \frac{\gamma(n)}{\binom{n}{2}} = \frac{1}{2}.\]
\end{theorem}
Interestingly, the structure of a size $n$ tanglegram with maximum tangle crossing number remains unknown.

We conclude with a polynomial time algorithm for computing lower
bounds on the tangle crossing number in Section~\ref{lowerb}. Drawing random tanglegrams of size $n$ from a uniform distribution, we give computational evidence that these lower bounds are $\Theta(n^2)$ with high probability, thus matching the result of Czabarka, Sz\'ekely, Wagner~\cite{CSW} that such a tanglegram has tangle crossing number $\Theta(n^2)$ with high probability.

\section{Preliminaries}\label{sec:pre}

Before delving into the proofs of our main theorems, we need to
establish some terminology and more formal definitions. A rooted
binary tree is a tree in which one vertex is designated as the root
and each vertex has either 0 or 2 children. A vertex with 0 children
is a \emph{leaf}. A vertex with 2 children is called an \emph{internal
  vertex}. Thinking of the root as a common ancestor to all other
vertices, the notions of \emph{descendant}, \emph{parent},
\emph{children} and \emph{sibling} become clear. If $B$ is a rooted
binary tree, a subset of the leaves of $B$ induces a binary subtree
$B'$, obtained from the smallest subtree of $B$ by suppressing all
degree 2 vertices and choosing as the root of $B'$ the vertex which
was closest to the root of $B$. For any internal vertex $v$ of $B$,
the subtree induced by the leaves which are descendants of $v$ is a
\emph{clade} of $B$ at $v$. If the two children of $v$ are leaves,
then the corresponding clade is called a \emph{cherry}.

A \emph{tanglegram layout} is a straight-line drawing in the plane of
two rooted binary trees, $L$ and $R$, each with $n$ leaves and a
perfect matching $M$ between their leaves (each leaf of $L$ paired
with a unique leaf of $R$) having the following properties:
\begin{itemize}
\item A plane drawing of $L$ appears in the half plane $x\leq 0$ with
  only the leaves of $L$ on the line $x=0$.
\item A plane drawing of $R$ appears in the half plane $x\geq 1$ with
  only the leaves of $R$ on the line $x=1$.
\item The matching is represented by a (straight-line) drawing of edges connecting each
  leaf of $L$ with the appropriate leaf of $R$.
\end{itemize}
The crossing number of such a layout is precisely the number of unordered pairs
of matching edges which cross. As there are $n$ matching edges, the
crossing number is clearly at most $\binom{n}{2}$.

To transform one layout into another, we define a \emph{switch}. First
observe that a layout induces a total order on the leaves of $L$ by
the $y$-coordinate of the leaves on the line $x=0$. Now each internal
vertex $v$ of $L$ has two children $v_1$ and $v_2$. To make a switch
at $v$, redraw the tree $L$ so that in the new layout, the order of
leaves $\ell$ and $\ell'$ is reversed if and only if one was a leaf
in the clade at $v_1$ and the other was a leaf in the clade at
$v_2$. The resulting tanglegram layout displays the new drawing of
$L$, an unchanged drawing of $R$, and the corresponding straight-line
drawing of the matching edges connecting the appropriate pairs of
leaves. Switch operations at internal vertices of $R$ are defined
analogously. Observe that the switch operation defines an equivalence
relation on the set of tanglegram layouts and each equivalence class
will be called a \emph{tanglegram}, denoted by the triple $(L,R,M)$.

Let $T=(L,R,M)$ be a tanglegram. The \textit{size} of $T$ is the size
of the matching $M$ (also the number of leaves in $L$ and the number
of leaves in $R$). The \emph{tangle crossing number} of $T$, denoted
$\crt(T)$, is the minimum number of pairs of edges that cross, among
all layouts of $T$. If $T$ has size $n$ then one can easily deduce
that $\crt(T)\leq \binom{n}{2}$.

Given a tanglegram $T=(L,R,M)$, a subset $M'$ of $M$ induces a
\emph{subtanglegram} $T'=(L',R',M')$ where $L'$ is the subtree of $L$
induced by leaves of $L$ which are endpoints of edges in $M'$ and $R'$
is defined similarly.

We let $\gamma(n)$ to denote the maximum tangle crossing number among all
tanglegrams of size $n$. In addition, we utilize the now standard notation $[n]$ for the set $\{1,2,\ldots, n\}$.

\section{Subtanglegrams of one size smaller}
\label{sec:nminusthree}

In a tanglegram of size $n$, the tangle crossing number is at most $\frac{1}{2}\binom{n}{2}$ (Theorem~\ref{thm:limit}). Given a tanglegram with tangle crossing number close to this upper bound, on average, each matching edge crosses 
one fourth of all the other matching edges. We explore the \emph{maximum}
number of crossings a single edge could contribute to the overall
tangle crossing number. Phrased another way, for any tanglegram $T$ of
size $n$ and subtanglegram $T'$ of size $n-1$, we determine the
maximum value of $\crt(T)-\crt(T')$.  The result is given in
Theorem~\ref{thm:edgeremoval}, an upper bound which
Theorem~\ref{thm:tight} shows to be tight, even for tanglegrams with
$T'$ planar. These two theorems together complete the proof of
Theorem~\ref{thm:edgerem}.

 Throughout this section, given
a tanglegram $T=(L,R,M)$ and $e\in M$, we use $T-e$ to denote the
subtanglegram of $T$ induced by edges in $M-e$.

\begin{theorem}
\label{thm:edgeremoval}
If $T = (L,R,M)$ is a tanglegram of size $n \geq 3$ and $e$ is any matching
edge of $T$, then \[\crt(T) - \crt(T-e) \leq n - 3.\]
\end{theorem}

\begin{proof} We will proceed by induction on $n$. First observe that
  if $T$ is a tanglegram of size at most $3$ then it is planar, and if $T$ is a tanglegram of size $4$ then $\crt(T)\le 1$~\cite{CSW2}; so the theorem
  is trivial when $n\le 4$.

  Let $n\ge 5$ and suppose that in every tanglegram of size $n-1$, each
  edge contributes at most $(n-1)-3$ to the tangle crossing number. Fix a tanglegram $T = (L,R,M)$ of size $n$, and let
  $e\in M$ be an arbitrary matching edge of $T$. Say $e$ has endpoints
  $u$ in $L$ and $v$ in $R$.  Fix an optimal layout $D'$ of
  $T - e = (L_{u}, R_{v}, M-e)$ with the fewest number of crossings.

  In $L$, let $w_{L'}$ be the parent of $u$ and let $L'$ be the clade at the other child of $w_{L'}$. (Similarly, define $w_{R'}$ and $R'$.)
  There are two planar drawings of $L$ whose subdrawings of $L_u$ agree with the drawing of $L_u$ in $D'$, one with $u$ immediately
  above the leaves of $L'$ and one with $u$ immediately below the leaves of $L'$. The ordering of the leaves of $L_u$ in each of these
  drawings of $L$ is exactly the same as the ordering of the leaves in the drawing of $L_u$ in $D'$. Further, one of these drawings of
  $L$ can be obtained from the other by making a switch at $w_{L'}$. A similar claim can be made about $R$, $v$, $R_v$, $w_{R'}$ and $R'$.
  Figure~\ref{fig:illustration} uses dashed lines to indicate the two potential positions of $u$ and for $v$ in a drawing of $T$.

    \begin{figure}[ht]
    \centering
    \input{figs/diagram.tikz}
    \caption{A figure that illustrates part of the proof of
      Theorem~\ref{thm:edgeremoval}}
    \label{fig:illustration}
  \end{figure}
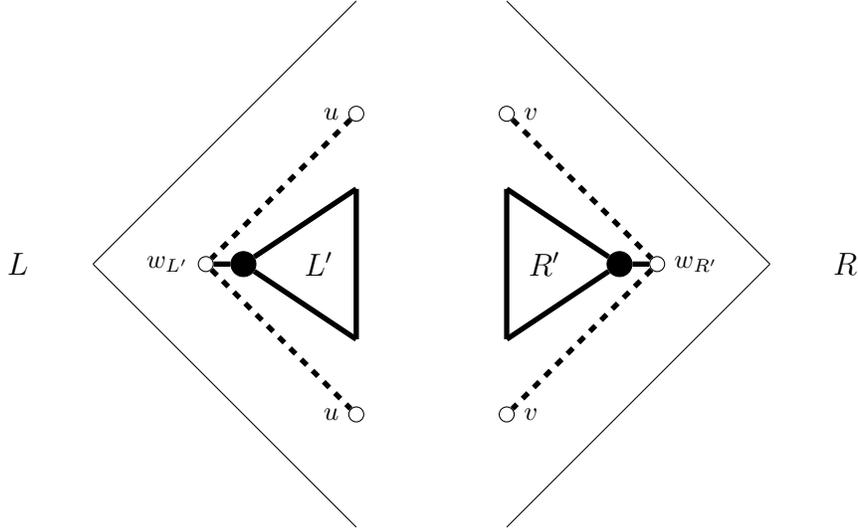

 We claim that there is drawing $D$ of $T$ using one of these two drawings of $L$ and one of these two drawings of $R$ in which matching edge $e$ crosses at most $n-3$ edges. This is sufficient to complete the proof as the number of crossings between two edges of $M-e$ in $D$ is exactly $\crt(T-e)$ (because the underlying drawing of $T-e$ remained unchanged) which implies $\crt(T) \leq \crt(D) \leq \crt(T-e) + (n-3)$ as desired.

 First observe that $L'$ and $R'$ each have at least one leaf. There are two cases to consider: (1) $L'$ and $R'$ each have exactly one leaf and they are matched in $M-e$, or (2) there is a leaf in $L'$ and a leaf in $R'$ which are not matched with one another.

   For the first case, let $f$ be the edge matching the single leaf in $L'$ with the single leaf in $R'$. Consider the drawing of $T$ with $u$ above $L'$ and $v$ above $R'$ so that   $e$ is above $f$. Suppose, for
  contradiction, that $e$ participates in strictly more than $n-3$
  crossings in this drawing of $T$. As $e$ does not cross itself or
  $f$, matching edge $e$ must cross every other edge in $M$.
  Since there are no leaves between the left endpoints of $e$ and $f$ and
  no leaves between the right endpoints of $e$ and $f$, it follows
  that $f$ also participates in $n - 2$ crossings in this drawing of
  $T$.  As the drawing of $T-e$ was optimal, we see that $f$
  contributes $n-2$ to the tangle crossing number of tanglegram $T-e$
  which had size $n-1$.  However, by the induction hypothesis, each edge in
  $T-e$ contributes at most $(n-1)-3$ crossings to $\crt(T-e)$, a
  contradiction.

  For the second case, let $u_{L'}$ be a leaf in $L'$ and $v_{R'}$ be a leaf in $R'$ which are not matched to each other. We say $u_{L'}$ (respectively, $v_{R'}$) is ``matched upward'' if the
  leaf to which it is matched is at least as high as the lowest leaf
  of $R^\prime$ (respectively, $L^\prime$). The leaf $u_{L'}$
  (respectively, $v_{R'}$) is ``matched downward'' if the leaf to
  which it is matched is no higher than the highest leaf of $R^\prime$
  (respectively, $L^\prime$).

  Let $f_1$ and $f_2$ be the matching edges, one with endpoint $u_{L'}$ and the other with endpoint $v_{R'}$. If $u_{L'}$ and $v_{R'}$
  are both matched upward (respectively, downward), draw the vertex
  $u$ below (respectively, above) $L^\prime$ and the vertex $v$ below
  (respectively, above) $R^\prime$.  On the other hand, if $u_{L'}$ is
  matched to a leaf higher (lower) than the leaves of $R'$ and
  $v_{R'}$ is matched to a leaf lower (higher) than the leaves of
  $L'$, then draw $u$ directly below (above) the leaves of $L'$ and
  $v$ directly above (below) the leaves of $R'$.  In each of these
  cases, the edge $e$ crosses neither $f_1$ nor $f_2$, and therefore
  crosses at most $n-3$ other edges, from which
  $\crt(T) - \crt(T - e) \leq n-3$ follows.
  \end{proof}

Now we prove that the inequality in Theorem~\ref{thm:edgeremoval} is
best possible. To do so, we present an infinite family of tanglegrams
$\{P_n: n\geq 4\}$ such that $P_n$ has size $n$, tangle crossing number $n-3$, and
there exists a matching edge $e$ such that $\crt(P_n-e)=0$. We say
$P_n$ is \emph{1-edge tangle planar} as $P_n$ is not planar but
there is a matching edge $e$ such that the subtanglegram $P_n-e$ is
planar. The two binary trees in $P_n$ are rooted caterpillars.

\begin{definition}
  The \emph{rooted caterpillar} $C_n$ of size $n$ is the unique rooted
  binary tree with $n$ leaves such that there are two leaves of
  distance $n-1$ from the root and for each $i\in [n-2]$ there is one
  leaf of distance $i$ from the root.
  (See Figure \ref{fig:c6} for an example.)
\end{definition}

\begin{figure}[ht]
\centering
\input{figs/C6.tikz}
\caption{The caterpillar $C_6$}
\label{fig:c6}
\end{figure}

\begin{definition}
  For each $n\ge 4$, we define the \emph{caterpillar tanglegram}
  $P_n = (L_n,R_n,M_n)$ as follows: $L_n$ and $R_n$ are copies of the
  rooted caterpillar $C_n$. We label the leaves of $L_n$ as $u_i$,
  where $i$ is the leaf's distance from the root. Since there are
  precisely two leaves at distance $n-1$, we arbitrarily label one of
  these $u_n$ instead. Similarly, the leaves of $R_n$ are labeled
  using $v_i$. Finally, we construct the matching
  $M_n = \{u_iv_{n-i} \mid i \in [n-1]\} \cup \{u_nv_n\}$. (See Figure
  \ref{fig:p8} for an example.)
\end{definition}

\begin{figure}[ht]
  \centering
  \input{figs/P8.tikz}
  \caption{The caterpillar tanglegram $P_8$}
  \label{fig:p8}
\end{figure}

\begin{theorem}\label{thm:tight}
  For each $n \geq 4$, the caterpillar tanglegram $P_n$ is 1-edge
  tangle planar and has tangle crossing number $n-3$.
\end{theorem}
\begin{proof}
  Note that the tanglegram $P_{n}-u_{n}v_{n}$ is clearly a planar
  tanglegram (see Figure \ref{fig:p8}), so $P_n $ is 1-edge tangle
  planar. The same drawing demonstrates that $\crt(P_n) \leq n-3$. It
  remains to show that $\crt(P_n) \geq n-3$. Suppose, for
  contradiction, that there is some $k$ for which $\crt(P_k) <
  k-3$. Furthermore, let $k$ be the least index witnessing this strict
  inequality. One can check by
  computer
  that $\crt(P_n) = n-3$ for
  $4 \leq n \leq 7$, so $k \geq 8$.
  Since $P_k$ contains a subdrawing of $P_7$, $\crt(P_k)\ge 7-3=4$.
  There are two cases for a
  fixed optimal drawing of $P_k$: at least one matching edge in the
  set $\{u_iv_{k-i} \mid 2 \leq i \leq k-2\}$ is part of a crossing or
  else none of them are.

  In the latter case, only the edges $u_1v_{k-1}$, $u_{k-1}v_1$, and
  $u_kv_k$ have crossings, and therefore $3\ge \crt(P_k) \ge 4$, a contradiction.

  In the former case, say the edge $u_jv_{k-j}$ is part of a
  crossing. The subtanglegram induced by $M_k- u_jv_{k-j}$ is
  isomorphic to $P_{k-1}$ and has tangle crossing number at most
  $\crt(P_k)-1$. It follows that
  $\crt(P_{k-1}) \leq \crt(P_k) - 1 < (k - 1) - 3$, which contradicts
  the minimality of $k$.
  \end{proof}

\section{Maximizing the crossing number}
\label{sec:onehalf}

While a single edge  in a tanglegram of size $n$ can contribute up to $n-3$ to the tangle crossing number, not all matching edges can realize this many crossings in a drawing which minimizes the tangle crossing number. The aim of this section is to better understand the maximum tangle crossing number among tanglegrams of the same size. To prove Theorem~\ref{thm:limit}. We begin with the first part:

\begin{theorem}\label{thm:onehalf}
  If $T$ is a tanglegram of size $n$ then
  $\crt(T)<\frac{1}{2} \binom{n}{2}.$  Consequently,
   $
  \limsup_{n\rightarrow\infty}
  \frac{\gamma(n)}{\binom{n}{2}}
  \leq
  \frac{1}{2}.
   $
\end{theorem}

\begin{proof}
  Let $T=(L,R,M)$ be a tanglegram.
  Suppose $\crt(T)=k$ and let $D$ be a tanglegram layout of
  $T$ having $k$ crossings.  By making a switch at every internal
  vertex in $R$, we obtain a new layout $D'$ of $T$.  Note that in
  $D'$, the plane drawing of $R$ can be viewed as a reflection of
  the drawing of $R$ in $D$ across the line $y=0$, while the plane
  drawing of $L$ is the same in both $D$ and $D'$.  For any unordered pair of
  edges $\{e,f\}$ in $M$, $e$ and $f$ cross in $D$ if and
  only if they do not cross in $D'$. This implies that $D'$ has
  exactly $\binom{n}{2}-k$ crossings. Since $\crt(T)=k$, every layout
  has at least $k$ crossings.  Consequently, $\binom{n}{2}-k\geq k$
  and $\crt(T)=k\le \frac{1}{2}\binom{n}{2}$.

  Suppose that, contrary to our statement, $\crt(T)=\frac{1}{2}\binom{n}{2}$.
  It follows from our proof so far that any layout of $T$ has $\frac{1}{2}\binom{n}{2}$ crossings, and for
  any unordered pair $\{e,f\}$ of matching edges there is a layout in which they cross.
  Let $C$ be a cherry of $R$ with leaves $\ell_1$ and $\ell_2$
  incident with matching edges $e,f\in M$. As noted above, $e$ and $f$
  must cross in some layout $D$ of $T$. From $D$, we create
  a new layout $D''$ by making a switch at the parent of $\ell_1$ and
  $\ell_2$.  The number of crossings in $D''$ is
  $\frac{1}{2}\binom{n}{2}-1$, a contradiction.
  \end{proof}

To complete the proof of Theorem~\ref{thm:limit}, we
prove
\[\liminf_{n\rightarrow\infty}\dfrac{\gamma(n)}{\binom{n}{2}}\geq
  1/2\] by constructing for each $n\geq 4$ a family $\cT_n$ of tanglegrams of
  size $n$ such that
for any $\varepsilon>0$ and large enough $n$, for all $T\in\cT_n$
 $\crt(T) / \binom{n}{2} \geq \frac{1}{2} - \varepsilon$.

We begin by constructing a family $\cT_{k^2}$ for each integer $k\ge 2$.
Any $T\in\cT_{k^2}$ is the result of the following procedure:
Take an arbitrary $(2k+2)$-tuple of size $k$ rooted binary trees
$(L_0,\ldots,L_k, R_0, \ldots, R_k)$.
Label the $k$ leaves of $L_0$ with labels $\{1,2,\ldots, k\}$ arbitrarily. For
each $i\in [k]$, identify the root of $L_i$ with leaf $i$ in $L_0$ and
assign labels $\{v_{ij}: j\in [k]\}$ to the $k$ leaves of $L_i$. The result
is the rooted binary tree $L$ with $k^2$ leaves. Similarly, $R$ is
built from $(R_0, R_1, \ldots, R_k)$ with leaf labels
$\{w_{ij}: i,j\in [k]\}$.  The matching is defined as
$M=\{v_{ij}w_{ji}: i,j\in [k]\}$.

Figure~\ref{fig:tau9} shows a tanglegram in $\cT_9$.  The binary
trees $L_1$, $L_2$, $L_3$ are marked by dashed rectangles. The tree
$L_{0}$ is the subtree of $L$ consisting of the roots of $L_{1}, L_{2}$, $L_{3}$,
and their ancestors. Note that the trees $L_0, L_1, L_2,$ and
$L_3$ need not be isomorphic. They are only isomorphic here because
there is only one binary tree, up to isomorphism, with 3 leaves.
Further, for any choice of two clades in $L$ and two clades in $R$,
there is at least one pair of edges which cross.

\begin{figure}[ht]
  \centering
  \input{figs/T9.tikz}
  \caption{A tanglegram in $\cT_9$}
  \label{fig:tau9}
\end{figure}

With a well-defined set of tanglegrams $\cT_{k^2}$ for each $k\geq 2$, we now define $\cT_n$ for any
integer $n$. Fix $n$ and choose $k$ such that $k^2 \leq n<(k+1)^2$.
Let $\cT_n$ be the set of
tanglegrams of size $n$ such that $T\in \cT_n$ if and only if there is a
tanglegram $T'\in \cT_{k^2}$ with $T'$ a subtanglegram of $T$.
Figure~\ref{fig:tau5} shows a tanglegram in $\cT_5$.  The tanglegram
with bold edges is a subtanglegram in $\cT_4$.

\begin{figure}[ht]
  \centering
  \input{figs/T5.tikz}
  \caption{A tanglegram in $\cT_{5}$}
  \label{fig:tau5}
\end{figure}

\begin{theorem}\label{th:liminf}
\[\liminf_{n\rightarrow\infty}\frac{\gamma(n)}{\binom{n}{2}}\geq\frac{1}{2}.\]
\end{theorem}
\begin{proof}
First we show that for each $k\ge 2$ and $T\in\cT_{k^2}$, $\crt(T)\ge\binom{k}{2}^2$.
Observe that for each $i\in [k]$, $L_{i}$ is a clade of $L$ at one of
the leaves of $L_{0}$. Therefore in any tanglegram layout of $(L,R,M)$
all the leaves of $L_{i}$ appear forming a vertical consecutive block,
for each $i\in [k]$. A similar assertion holds for the leaves of
$R_{i}$, $i\in[k]$.  For any $\{i, j\}, \{a, b\}\subseteq [k]$, there
are 4 edges with both endpoints in the clades $L_i$, $L_j$, $R_a$, and
$R_b$. Because the leaves in a single clade form a vertical
consecutive block in any layout, either the edges $v_{i a}w_{ai}$ and
$v_{jb}w_{bj}$ or the edges $v_{ja}w_{aj}$ and $v_{ib}w_{bi}$ form a
crossing. As a result, $\crt(T) \geq \binom{k}{2}^2$.

As the tangle crossing number of each tanglegram in $\cT_{k^2}$ is
at least $\binom{k}{2}^2$, the tangle crossing number of each
tanglegram in $\cT_n$ with $n>k^2$ is also at least $\binom{k}{2}^2$.

 Let $n\ge 4$ and $k=\lfloor\sqrt{n}\rfloor$, so  $k^2 \leq n < (k+1)^2$. Observe that for each
  tanglegram $T\in \cT_n$, $\crt(T) \geq \binom{k}{2}^2$. Therefore
    \begin{align*}
      \frac{\gamma(n)}{\binom{n}{2}}
      \geq \dfrac{\displaystyle\max_{T\in \cT_n} \crt(T)}{\binom{n}{2}}
      \geq\frac{\binom{k}{2}^2}{\binom{(k+1)^2}{2}}
      &  = \frac{1}{2} \left(1-\frac{2}{k+2}\right) \left(1-\frac{2}{k+1}\right)^2\\
      & \geq  \frac{1}{2} \left(1-\frac{2}{\sqrt{n}+1}\right) \left(1-\frac{2}{\sqrt{n}}\right)^2.
      \end{align*}

As a result,
\[ \liminf_{n\rightarrow\infty} \frac{\gamma(n)}{\binom{n}{2}} \geq
  \liminf_{n\rightarrow\infty} \frac{1}{2}
  \left(1-\frac{2}{\sqrt{n}+1}\right)
  \left(1-\frac{2}{\sqrt{n}}\right)^2 =\frac{1}{2}.\]
\end{proof}

Theorems~\ref{thm:onehalf} and~\ref{th:liminf} complete the proof of
Theorem~\ref{thm:limit}.

\section{Lower bound of the tangle crossing number} \label{lowerb}

Let $T=(L,R,M)$ be a tanglegram of size $n$. In this section, we present an algorithm which outputs a non-trivial lower bound for the tangle crossing number of $T$ in $O(n^4)$ time. As we will show, this lower bound is tight for some tanglegrams with quadratic tangle crossing number. The algorithm runs in two phases. First it partitions the leaves of each tree into clades. In the second phase the clades are used to compute the lower bound for $\crt(T)$. Now we describe the algorithm for partitioning the leaves a given tree into clades, given a restriction on their size. Note that we use this algorithm independently for $L$ and $R$.

\begin{algorithm}
\caption{Partition leaves into clades}\label{alg:clade_partition}
\begin{algorithmic}[1]
\Input A binary tree $B$ and a number $C>1$.
\Output A partition of the leaves of $B$ into clades of size at most $C$.
\State Label each vertex $v$ in $B$ with the number of leaves in the clade of $B$ at $v$.
\State Let $\{v_i\}_{i=1}^k$ be the set of vertices such that the label at $v_i$ is at most $C$ and whose parent has label greater than $C$.
\State \textbf{return} $\{V_i\}_{i=1}^k$, where $V_i$ is the set of leaves in the clade of $B$ at $v_i$.
\end{algorithmic}
\end{algorithm}

Algorithm~\ref{alg:clade_partition} can be implemented in $O(n)$ time. This follows from noting that step 1 requires a post-order traversal of $B$ and each of steps 2 and 3 require a pre-order traversal of $B$. Let $W=\{v_i\}_{i=1}^k$ be the set of vertices from step~2. Note that if $v_i,v_j\in W$, then $v_i$ is not an ancestor of $v_j$ and vice versa. It is easy to see that a consequence of this property is that the collection $\{V_i\}_{i=1}^k$ from step~3 is a partition of the leaves of $B$ into clades. Algorithm~\ref{alg:lower_bound} below computes the lower bound for the tangle crossing number.

\begin{algorithm}
\caption{Tangle crossing number lower bound}\label{alg:lower_bound}
\begin{algorithmic}[1]
\Input Tanglegram $T=(L,R,M)$, and numbers $C_L,C_R>1$.
\Output A lower bound for $\crt(T)$.
\State Use Algorithm~\ref{alg:clade_partition} to obtain $\{U_i\}_{i=1}^\ell$ for $L$ and $C_L$.
\State Use Algorithm~\ref{alg:clade_partition} to obtain $\{V_i\}_{i=1}^r$ for $R$ and $C_R$.
\State For each $U_i$ and $V_j$, let $M_{i,j}$ be the set of matching edges with one endpoint in $U_i$ and one in $V_j$.
\State \textbf{return}
$\sum_{i_1,i_2}\sum_{j_1,j_2}\min\{|M_{i_1,j_1}||M_{i_2,j_2}|,|M_{i_1,j_2}||M_{i_2,j_1|}\},$
where the sum is over $\{i_1,i_2\} \subseteq [\ell]$ and $\{j_1,j_2\} \subseteq [r]$.
\end{algorithmic}
\end{algorithm}

Note that Algorithm~\ref{alg:lower_bound} runs in $O(n^4)$. This follows since steps~1 and~2 take $O(n)$ time, step~3 takes $O(n^2)$ time, and step~4 takes $O(n^4)$ time.
 To prove correctness, suppose $U_a$ and $U_b$ are clades in $L$ and suppose $V_c$ and $V_d$ are clades in $R$. Because these are clades, any layout of $T$ will have either the $M_{ac}$
edges cross the $M_{bd}$ edges or the $M_{ad}$ edges cross the
$M_{bc}$ edges. As a result, these 4 clades will contribute at least
$\min\{|M_{ac}||M_{bd}|, |M_{ad}||M_{bc}|\}$ to the tangle crossing
number. Thus, as done in step~4, summing these minimums over all $\binom{\ell}{2}$ pairs of clades
from $L$ and $\binom{r}{2}$ pairs of clades from $R$, we obtain a
lower bound on $\crt(T)$.

One may notice that Algorithm~\ref{alg:lower_bound} depends on the choice of $C_L$ and $C_R$. When $n=k^2$, the choice of $C_L=C_R=\sqrt{n}$ is optimal for the tanglegrams in $\cT_{k^2}$ from Section~\ref{sec:onehalf} described for the proof of Theorem~\ref{th:liminf}. For each tree in these tanglegrams, Algorithm~\ref{alg:clade_partition} finds the $k$ clades with $k$ leaves that were used to build these trees. With this clade partition, $M_{i,j}=1$ for all $i,j\in [k]$. So the tangle crossing number is at least $\binom{k}{2}^2$ by Algorithm~\ref{alg:lower_bound}.  It is not hard to find tanglegrams in $\cT_{k^2}$ with tangle crossing number exactly $\binom{k}{2}^2$. Thus the output of Algorithm~\ref{alg:lower_bound} for the family of tanglegrams $\cT_{k^2}$ is tight.

We ran simulations for different choices of $C_L$ and $C_R$ with random tanglegrams drawn from a uniform distribution. Figure~\ref{fig:simulation} shows the average lower bounds when $C_L,C_R\in\{4,\sqrt{n},\frac{n}{2}\}$. For each $n\in\{10,\ldots,100\}$, we picked $100$ tanglegrams of size $n$ uniformly at random. The random sampling algorithm is a SageMath~\cite{sagemath} implementation of Algorithm~3 from~\cite[p. 253]{BKM17}. The source code for our implementation is available at~\cite{BL17}.  Based on the simulations, it appears that $C_L=C_R=\frac{n}{2}$ yields better lower bounds.

\begin{figure}[ht]
\includegraphics{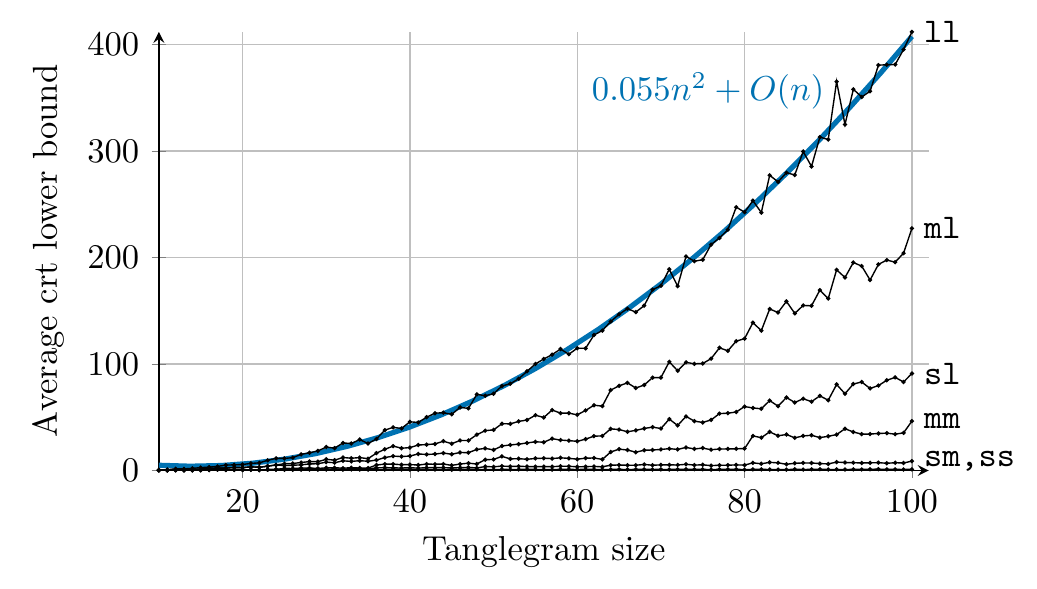}
\caption{The average lower bound for $\crt(T)$ for different choices of $C_L$ and $C_R$. The symbols \texttt{s}, \texttt{m} and \texttt{l} represent $4,\sqrt{n},$ and $n/2$ respectively. The curve labeled \texttt{ml} represents the average output with $C_L=\sqrt{n}$ and $C_R=n/2$.}
\label{fig:simulation}
\end{figure}
In \cite{CSW} it is shown that there exists $C>0$ such that a random tanglegram has tangle crossing number $C n^2$ with high probability.  Fitting the \texttt{ll} curve from Figure~\ref{fig:simulation}, the curve corresponding to $C_L=C_R=n/2$, to a quadratic function via least squares yields $0.055 n^2+O(n)$. This suggests that the tangle crossing number of the random tanglegram is at least $0.055n^2$. For the same sample, a plot of the maximum lower bounds is fit by the curve $0.08n^2 + O(n)$. These two growth rates are to be compared with the upper bound of $0.25n^2$ from Theorem~\ref{thm:limit}.

Another way to view this process is to create an auxiliary bipartite multigraph with a vertex for each clade and the number of edges between two clades is the number of edges which match a vertex of one clade to a vertex of the other clade. We then restrict to straight-line drawings where the vertices of one partite set remain on the line $x=0$ and the vertices of the other partite set lie on the line $x=1$. The minimum crossing number over all such drawings of this multigraph provides a lower bound on the crossing number of the tanglegram. However, Garey and Johnson~\cite{GJ} proved that even this problem on the auxiliary bipartite multigraph is NP-complete.

\section{Open Questions and Further Work}
Although the lower bound provided in Section 5 is tight for many small tanglegrams, we don't expect it being close to the real answer all the time, since we are doing a polynomial time approximation to an NP-hard problem.
One may notice that the lower bound is dependent on the
choice of clades. While we made an arbitrary choice, we are interested
in polynomial time algorithms to choose the clades for an optimized
lower bound.

In Section~\ref{sec:onehalf}, we provided a family of tanglegrams with
crossing number asymptotically $\frac{1}{2}\binom{n}{2}$. While the
tangle crossing number of tanglegrams in $\cT_n$ is at least
$\binom{\lfloor \sqrt{n} \rfloor}{2}^2$,
there are tanglegrams of size $n$ with larger tangle crossing
number.
Is it perhaps true that $\max\{\crt(T):T\in\cT_n\}=\gamma(n)$, at least for $n=k^2$?
We remain interested in the maximum tangle crossing
number over all tanglegrams of size $n$.

\section{Acknowledgements}
The authors would like to extend their gratitude to the American
Mathematical Society for organizing the Mathematics Research Community
workshops where this work began. All authors were
supported by the National Science Foundation under Grant Number DMS
1641020. Smith
was also supported in part by NSF-DMS grant 1344199 and Sz\'ekely was also supported by the NSF-DMS grants
1300547 and 1600811.
Da Lozzo was supported by the U.S.~Defense Advanced
Research Projects Agency (DARPA) under agreement no.~AFRL
FA8750-15-2-0092.
The views expressed are those of the authors and do not reflect the
official policy or position of the Department of Defense or the
U.S.~Government.

\bibliographystyle{plain}
\bibliography{refs_short}

\end{document}

%% file: figs/diagram.tikz
\begin{tikzpicture}[scale=0.5,
every picture/.style={thick}]
\tikzstyle{none}=[inner sep=0pt]
\tikzstyle{invisible}=[circle,fill=Black,draw=Black]
\tikzstyle{simple}=[-,draw=Black,line width=2.000]
		\node [style=none] (0) at (0, 0) {};
		\node [draw,circle,black,inner sep=2pt,label=west:$u$] (1) at (0, 3) {};
		\node [draw,circle,black,inner sep=2pt,label=west:$u$] (2) at (0, 11) {};
		\node [style=none] (3) at (0, 14) {};
		\node [style=none] (4) at (-7, 7) {};
		\node [draw,circle,black,inner sep=2pt,label=west:$w_{L'}$] (5) at (-4, 7) {};
		\node [style=none] (6) at (0, 9) {};
		\node [style=invisible] (7) at (-3, 7) {};
		\node [style=none] (8) at (0, 5) {};
		\node [style=none] (9) at (-1, 7) {{\Large$L'$}};
		\node [draw,circle,black,inner sep=2pt,label=east:$v$] (10) at (4, 3) {};
		\node [style=none] (11) at (4, 9) {};
		\node [style=none] (12) at (5, 7) {{\Large$R'$}};
		\node [draw,circle,black,inner sep=2pt,label=east:$w_{R'}$] (13) at (8, 7) {};
		\node [style=none] (14) at (4, 14) {};
		\node [style=none] (15) at (4, 5) {};
		\node [style=none] (16) at (4, 0) {};
		\node [draw,circle,black,inner sep=2pt,label=east:$v$] (17) at (4, 11) {};
		\node (18) at (11, 7) {};
		\node [style=invisible] (19) at (7, 7) {};
		\node (20) at (-9, 7) {{\Large$L$}};
		\node (21) at (13, 7) {{\Large$R$}};
		
		\draw (3.center) to (4.center);
		\draw (4.center) to (0.center);
		\draw [dashed,style=simple] (2) to (5);
		\draw [dashed,style=simple] (5) to (1);
		\draw [style=simple] (6.center) to (7);
		\draw [style=simple] (7) to (8.center);
		\draw [style=simple] (8.center) to (6.center);
		\draw [style=simple] (7) to (5);
		\draw (14.center) to (18.center);
		\draw (18.center) to (16.center);
		\draw [dashed,style=simple] (17) to (13);
		\draw [dashed,style=simple] (13) to (10);
		\draw [style=simple] (11.center) to (19);
		\draw [style=simple] (19) to (15.center);
		\draw [style=simple] (15.center) to (11.center);
		\draw [style=simple] (19) to (13);
 
\end{tikzpicture}

%% file: figs/C6.tikz
\begin{tikzpicture}[yscale=0.8,scale=0.4,every node/.style={draw,circle,black,inner sep=2pt},
every picture/.style={thick},
leavesnode/.style={circle, draw, inner sep=1pt}]
\node (0) at (0, 0) {};
\node (1) at (0, 2) {};
\node (2) at (-1, 1) {};
\node (3) at (0, 4) {};
\node (4) at (-2, 2) {};
\node (5) at (0, 6) {};
\node (6) at (-3, 3) {};
\node (7) at (-3, 3) {};
\node (8) at (0, 8) {};
\node (9) at (-4, 4) {};
\node (10) at (0, 10) {};
\node (11) at (-5, 5) {};

\draw (11) -- (9);
\draw (9) -- (6);
\draw (6) to (4);
\draw (4) to (2);
\draw (2) to (0);
\draw (2) to (1);
\draw (4) to (3);
\draw (6) to (5);
\draw (9) to (8);
\draw (11) to (10);
\end{tikzpicture}

%% file: figs/P8.tikz
\begin{tikzpicture}[yscale=0.8,scale=0.5,every node/.style={draw,circle,black,inner sep=2pt},
every picture/.style={thick},
leavesnode/.style={circle, draw, inner sep=1pt}]
\begin{scope}
\node[label=west:$u_7$] (0) at (0, 0) {};
\node[label=west:$u_8$] (1) at (0, 2) {};
\node[label=west:$u_6$] (3) at (0, 4) {};
\node[label=west:$u_5$] (5) at (0, 6) {};
\node[label=west:$u_4$] (8) at (0, 8) {};
\node[label=west:$u_3$] (10) at (0, 10) {};
\node[label=west:$u_2$] (11) at (0, 12) {};
\node[label=west:$u_1$] (12) at (0, 14) {};
\node[label=east:$v_8$] (16) at (3, 12) {};
\node[label=east:$v_4$] (17) at (3, 6) {};
\node[label=east:$v_5$] (22) at (3, 8) {};
\node[label=east:$v_7$] (26) at (3, 14) {};
\node[label=east:$v_6$] (27) at (3, 10) {};
\node[label=east:$v_3$] (28) at (3, 4) {};
\node[label=east:$v_1$] (30) at (3, 0) {};
\node[label=east:$v_2$] (31) at (3, 2) {};
\end{scope}

\node (2) at (-1, 1) {};
\node (4) at (-2, 2) {};
\node (6) at (-3, 3) {};
\node (7) at (-3, 3) {};
\node (9) at (-4, 4) {};
\node (13) at (-5, 5) {};
\node (14) at (-6, 6) {};
\node (15) at (-7, 7) {};
\node (18) at (8, 9) {};
\node (19) at (4, 13) {};
\node (20) at (6, 11) {};
\node (21) at (5, 12) {};
\node (23) at (7, 10) {};
\node (24) at (9, 8) {};
\node (25) at (10, 7) {};
\node (29) at (6, 11) {};

\draw (12) to (15);
\draw (15) to (14);
\draw (14) to (13);
\draw (13) to (9);
\draw (9) to (6);
\draw (6) to (4);
\draw (4) to (2);
\draw (2) to (0);
\draw (2) to (1);
\draw (4) to (3);
\draw (6) to (5);
\draw (9) to (8);
\draw (13) to (10);
\draw (14) to (11);
\draw (30) to (25);
\draw (25) to (24);
\draw (24) to (18);
\draw (18) to (23);
\draw (23) to (29);
\draw (29) to (21);
\draw (21) to (19);
\draw (19) to (26);
\draw (19) to (16);
\draw (21) to (27);
\draw (29) to (22);
\draw (23) to (17);
\draw (18) to (28);
\draw (24) to (31);
\draw (12) to (26);
\draw (0) to (30);
\draw (16.south west) to (1.north east);
\draw (3) to (31);
\draw (5) to (28);
\draw (8) to (17);
\draw (10) to (22);
\draw (11) to (27);
\end{tikzpicture}

%% file: figs/T9.tikz
\begin{tikzpicture}[yscale=0.8,scale=0.5,every node/.style={draw,circle,black,inner sep=2pt},
every picture/.style={thick},
leavesnode/.style={circle, draw, inner sep=1pt}]

%%%Middle vertices and matching
\foreach \i in {1,2,3}{
\foreach \j in {1,2,3}{
\pgfmathsetmacro{\y}{int(24-6*\i-2*\j)}
\node [label={[label distance=-1pt]180:$v_{\i\j}$}](v\i\j) at (0,\y) {};
\node [label={[label distance=-1pt]0:$w_{\i\j}$}](w\i\j) at (4,\y) {};
}}
\foreach \i in {1,2,3}{
\foreach \j in {1,2,3}{
\draw (v\i\j) -- (w\j\i);
}}

%%%Left tree
\node (L1m) at (-1,15) {};
\node (L2m) at (-1,7) {};
\node (L3m) at (-1,1) {};
\node (L1r) at (-2,14) {};
\node (L2r) at (-2,8) {};
\node (L3r) at (-2,2) {};
\node (L0m) at (-4,5) {};
\node (L0r) at (-6,8) {};

\draw (L1m) -- (v11);
\draw (L1m) -- (v12);
\draw (L2m) -- (v22);
\draw (L2m) -- (v23);
\draw (L3m) -- (v32);
\draw (L3m) -- (v33);
\draw (L1r) -- (L1m);
\draw (L1r) -- (v13);
\draw (L2r) -- (v21);
\draw (L2r) -- (L2m);
\draw (L3r) -- (v31);
\draw (L3r) -- (L3m);
\draw (L0m) -- (L2r);
\draw (L0m) -- (L3r);
\draw (L0r) -- (L1r);
\draw (L0r) -- (L0m);

%%%Right tree
\node (R1m) at (5,13) {};
\node (R2m) at (5,7) {};
\node (R3m) at (5,1) {};
\node (R1r) at (6,14) {};
\node (R2r) at (6,8) {};
\node (R3r) at (6,2) {};
\node (R0m) at (8,11) {};
\node (R0r) at (10,8) {};

\draw (R1m) -- (w12);
\draw (R1m) -- (w13);
\draw (R2m) -- (w22);
\draw (R2m) -- (w23);
\draw (R3m) -- (w32);
\draw (R3m) -- (w33);
\draw (R1r) -- (w11);
\draw (R1r) -- (R1m);
\draw (R2r) -- (w21);
\draw (R2r) -- (R2m);
\draw (R3r) -- (w31);
\draw (R3r) -- (R3m);
\draw (R0m) -- (R1r);
\draw (R0m) -- (R2r);
\draw (R0r) -- (R0m);
\draw (R0r) -- (R3r);

%dashed rentangles
\draw [thin,dashed] (-2.5,11.5) rectangle (0.5,16.5);
\draw [thin,dashed] (-2.5,5.5) rectangle (0.5,10.5);
\draw [thin,dashed] (-2.5,-0.5) rectangle (0.5,4.5);
\draw [thin,dashed] (3.5,11.5) rectangle (6.5,16.5);
\draw [thin,dashed] (3.5,5.5) rectangle (6.5,10.5);
\draw [thin,dashed] (3.5,-0.5) rectangle (6.5,4.5);
\begin{scope}[every node/.style={rectangle}]
\node[left] at (-2.6,15) {$L_1$};
\node[left] at (-2.6,9) {$L_2$};
\node[left] at (-2.6,1) {$L_3$};
\node[right] at (6.6,15) {$R_1$};
\node[right] at (6.6,7) {$R_2$};
\node[right] at (6.6,1) {$R_3$};
\end{scope}

\end{tikzpicture}

%% file: figs/T5.tikz
\begin{tikzpicture}[yscale=0.8,scale=0.5,every node/.style={draw,circle,black,inner sep=2pt},
every picture/.style={thick},
leavesnode/.style={circle, draw, inner sep=1pt}]
%%%Middle vertices
\foreach \i in {1,...,5}{
\pgfmathsetmacro{\y}{int(10-2*\i)}
\node (v\i) at (0,\y) {};
\node (w\i) at (2,\y) {};
}
%%%All other vertices
\node (v6) at (-1,1) {};
\node (v7) at (-1,5) {};
\node (v8) at (-3,3) {};
\node (v9) at (-4,4) {};
\node (w6) at (3,1) {};
\node (w7) at (3,7) {};
\node (w8) at (4,2) {};
\node (w9) at (6,4) {};

%%%Bolded edges
\begin{scope}[line width=2pt]
\draw (v6) -- (v8) -- (v7);
\draw (v2) -- (v7) -- (v3);
\draw (v4) -- (v6) -- (v5);
\draw (w1) -- (w7) -- (w2);
\draw (w7) -- (w9) -- (w8);
\draw (w3) -- (w8) -- (w6);
\draw (w6) -- (w5);
\foreach \i/\j in {2/1,3/3,4/2,5/5}{
\draw (v\i) -- (w\j);
}
\end{scope}
%%%Other edges
\draw (v1) -- (w4);
\draw (v1) -- (v9) -- (v8);
\draw (w4) -- (w6);

%%%Marks
% \begin{scope}[every node/.style={rectangle,draw=none}]
% \node[left] at (-0.1,8) {$v$};
% \node[right] at (2.1,2) {$w$};
% \node[left] at (-4.1,4) {new root};
% \node[right] at (3.1,1) {subdivision vertex};
% \end{scope}
\end{tikzpicture}